\documentclass[twoside,12pt]{amsart}
\usepackage{graphicx}
\usepackage{amssymb}
\usepackage{amsthm}
\usepackage{amsmath}
\newfont{\bbb}{msbm10 scaled\magstep 1}

\DeclareMathOperator{\conv}{conv}

\DeclareMathOperator{\bd}{bd}
\DeclareMathOperator{\lin}{lin}
\newcommand{\Eu}{\mathbb{E}}
\renewcommand{\Re}{\mathbb{R}}

\newcommand{\B}{\mathbf B}
\newcommand{\Sph}{\mathbb{S}}
\newcommand{\X}{\mathbb{X}}
\newcommand{\XX}{\mathfrak{X}_o}

\newcommand{\NS}{(\X,\|\cdot\|)}
\newtheorem{thm}{Theorem}[section]
\newtheorem{cor}{Corollary}[section]
\newtheorem{rem}{Remark}[section]
\newtheorem{ex}{Example}[section]

\newtheorem{ddd}{Definition}[section]
\newtheorem{problem}{Problem}
\newtheorem{ques}{Question}
\newtheorem{pro}{Proposition}[section]

\begin{document}

\title[Semi-inner product spaces]{Semi-inner products and the concept of semi-polarity}
\author[\'A. G.Horv\'ath, Z. L\'angi and M. Spirova]{\'Akos G.Horv\'ath, Zsolt L\'angi and Margarita  Spirova}

\address{\'Akos G.Horv\'ath\\
Dept. of Geometry, Budapest University of Technology,
Egry J\'ozsef u. 1., Budapest, Hungary, 1111}
\email{ghorvath@math.bme.hu}
\address{Zsolt L\'angi\\
Dept. of Geometry, Budapest University of Technology,
Egry J\'ozsef u. 1., Budapest, Hungary, 1111}
\email{zlangi@math.bme.hu}
\address{Margarita Spirova\\
Fakultat f\"ur Mathematik, TU Chemnitz,
D-09107, Chemnitz, Germany}
\email{margarita.spirova@mail.de}

\thanks{The authors gratefully acknowledge the support of the J\'anos Bolyai Research Scholarship of the Hungarian Academy of Sciences.}

\begin{abstract}
The lack of an inner product structure in Banach spaces yields the motivation to introduce a semi-inner product with a
more general axiom system, one missing the requirement for symmetry, unlike the one determing a Hilbert space.
We use it on a finite dimensional real Banach space $(\X, \| \cdot\|)$ to define and investigate three concepts.
First, we generalize that of \emph{antinorms}, already defined in Minkowski planes, for even dimensional spaces.
Second, we introduce \emph{normality maps}, which in turn leads us to the study of \emph{semi-polarity}, a variant of the notion of polarity, which makes use of the underlying semi-inner product.
\end{abstract}

\maketitle

\section{Introduction}\label{sec:intro}

Motivated by the lack of inner product in general Banach spaces, Lumer \cite{Lu} defined semi-inner product
spaces, which enabled him to adapt Hilbert space arguments to the theory of Banach spaces.
From the viewpoint of functional analysis, real (and complex) semi-inner product spaces have been in the mainstream of scientific research;
for references in this regard see the book \cite{Dragomir}.
Our aim is to examine them for purely geometric purposes.
We start with some preliminary definitions.

Let $\X$ be a real vector space. A \emph{semi-inner product} on $\X$ is  a real function $[\cdot, \cdot]$ on $\X\times \X$ satisfying the following
properties for any $x,y,z \in \X$.
\begin{enumerate}
\item[(i)] $[x+y, z]=[x, z]+[y, z]$, $[\lambda x, y]=\lambda [x, y]$ for all real $\lambda$,
\item[(ii)] $[x, x]> 0$, when $x\neq 0$,
\item[(iii)] $[x, y]^2\leq [x, x][y, y]$.
\end{enumerate}

A real vector space $\X$, equipped with a semi-inner product, is said to be a (real) \emph{semi-inner product space}.
It is well-known that any semi-inner product $[\cdot, \cdot]$ on $\X$ induces a norm, by setting $\|x\|=\sqrt{[x, x]}$.
Conversely, every \emph{Banach space} $\NS$  can be turned  into a semi-inner product space (see \cite[Theorem 1]{Gi}) in the following way.

Let $\Sph:=\{x\in \X: \|x\|=1\}$ be the \emph{unit sphere} of $\NS$, and $\X^\ast$ be the \emph{dual space of} $\X$. On
$\X^\ast$ one  can define a norm $\|\cdot\|^\ast$,  called the \emph{dual norm}, in the usual way, i.e.,
\begin{equation}\label{1}
\|f\|^\ast:=\sup\{f(x): \|x\|=1\}\;\;\;\text{for}\;\;\;f\in \X^\ast .
\end{equation}
If  $\Sph^\ast$ is the unit sphere of  $(\X^\ast, \|\cdot\|^\ast)$,
then, by the Hahn-Banach Theorem, for any $x\in \Sph$ there exists at least one functional  (exactly one functional if the norm is smooth) $f_x\in
\Sph^\ast$ with $f_x(x)=1$. For any $\lambda x\in\X$, where $x\in \Sph$, we choose $f_{\lambda x}\in\X^\ast$ such that $f_{\lambda x}=\lambda f_x$.
Then a semi-inner product $[\cdot, \cdot]$ is defined on $\X$ by
\begin{equation}
[x, y]:=f_y(x).\label{3}
\end{equation}

The aim of the paper is to investigate three geometric concepts related to semi-inner products.
We collect the main tools of our examination in Section~\ref{sec:prelim}, then, in Section~\ref{sec:antinorms},
we introduce and investigate the \emph{antinorm} of an even dimensional real Banach space by means of a symplectic form defined on the space.
For normed planes, this notion was studied, e.g. in \cite{Ma-Sw2} and \cite{Bus}.
In Section~\ref{sec:normalitymaps}, by means of antinorms, we define and examine \emph{normality maps}.
In Section~\ref{sec:semipolar}, based on the semi-inner product structure of $\X$, we define the notion of \emph{semi-polars} in $\X$
and generalize the properties of polars, known in Euclidean spaces. Finally, in Section~\ref{sec:remarks} we collect our questions and additional remarks.

In functional analysis the polar of a set in a space $\X$ is a subset of the dual space $\X^\ast$ (cf. \cite{Al-To}),
in geometry polarity is regarded as a correspondence between sets of the same Euclidean space,
where linear functionals in $\X^\ast$ are identified with points in $\X$ via the inner product.
In our variant of polarity, we provide a correspondence between subsets of the same normed space,
based on the semi-inner product defined by the norm.

\section{Preliminaries}\label{sec:prelim}

Let $(\X, \| \cdot\|)$ be a \emph{normed  space} (i.e., a finite dimensional real Banach space) with \emph{origin} $o$ and
\emph{unit ball} $\B=\{x\in\X: \|x\|\leq 1\}$, which is a compact, convex subset of $\X$ with boundary $\Sph$, centered at its interior point $o$.
Let $\B_{\Eu}$ and $\Sph_{\Eu}$  be the  unit ball and sphere, respectively, with respect to a \emph{Euclidean norm},
i.e., a norm induced by an inner product on $\X$.
A vector $x\neq 0$ is \emph{normal} to a vector $y\neq 0$, denoted by $x\dashv y$, if, for any real $\lambda$, the  inequality
$\|x\|\leq \|x+\lambda y\|$ holds; see, e.g. \cite[$\S$ 6]{Ma-Sw-Wei}.

For a convex body $K$, i.e., a compact, convex subset of $\X$ with nonempty interior and $u\neq o$, let $h(K, u)$ be the \emph{support function}
in direction $u$. The \emph{support function  of $K$ with respect to the norm } $\|\cdot\|$ is defined by $\displaystyle h_B(K, u)=\frac{h(K,
u)}{h(\B , u)}$. Alternatively, for  every $u\neq o $ this normed support function  $h_B(K, u)$ can be viewed as the signed distance with respect
to $\|\cdot\|$ from the origin $o$ to a supporting hyperplane $H$ of $K$ such that  the outer normal of $H$ with respect to $K$ yields a
positive inner product with $u$; see, e.g. \cite{Cha-Groe} or \cite[$\S$ 2]{Ma-Sw1}. This means that the normed support function $h_B(K, u)$ of
$K$ can be expressed as $\sup\{[x, u]: x\in K\}$.

We denote the family of all convex bodies, containing the origin $o$ as an interior point, by $\XX$. For $K\in\XX$, let $g(K,\cdot)$ be
the \emph{gauge function} of $K$, i.e.,
\[
g(K, x):=\min\{\lambda\geq 0: x\in \lambda K\}\;\;\;\text{for}\;\;\; x\in\X.
\]
Note that $g(\B, x)=\|x\|$ for every $x\in \X$.

From now on, let $\NS$ be a smooth and strictly convex normed space.
We denote by $[\cdot, \cdot]$ the semi-inner product induced by the norm $\|\cdot\|$.
If $\NS$ is an inner product space, i.e., the corresponding semi-inner product is, in addition, symmetric, then we denote this product by $[\cdot, \cdot]_{\Eu}$.   The following properties are proved in \cite{Gi} (see also \cite{Ko}, \cite[$\S$ 2.4]{Ho}, and \cite{La}).

\begin{enumerate}
\item[(iv)] The homogeneity property: $[x, \lambda y]=\lambda[x, y]$ for all $x, y\in \X$ and all real $\lambda$.
\item[(v)] $[y, x]=0 \Longleftrightarrow \|x\|\leq\|x+\lambda y\|$ for all $\lambda\in\Re$.
\item[(vi)] The generalized Riesz-Fischer Representation Theorem: To every linear functional $f\in \X^\ast$ there exists a unique vector
    $y\in \X$ such that $f(x)=[x, y]$ for all $x\in\X$. Then $[x, y]=[x, z]$ for all $x\in\X$ if and only if $y=z$.
\item[(vii)] The dual vector space $\X^\ast$ is a semi-inner product space  by $[f_x, f_y]^\ast=[y, x]$.
\end{enumerate}

\begin{rem}\label{rem3}
Property (v) can be written in the form
\begin{enumerate}
\item[(v \hspace{-1pt}$'$)] $x\neq 0$, $y\neq 0$ and $[y, x]=0\Longleftrightarrow x\dashv y.$
\end{enumerate}
\end{rem}

\begin{rem}\label{rem1}
By Property (vi), we have a one-to-one map $F: \X\rightarrow \X^\ast$ with $F: x \mapsto f_x$, where $f_x$ is determined
by (\ref{3}). Property (vii) implies that $F$ is norm-preserving.
\end{rem}


\begin{pro}\label{pro2}
The norm defined by (\ref{1}) is induced by the semi-inner product $[\cdot, \cdot]^\ast$ on $\X^\ast$ defined by (vii).
\end{pro}

\begin{proof}
First, observe that
\begin{equation}\label{4}
\sqrt{[f_x, f_x]^\ast}=\sqrt{[x, x]}=\|x\| ,
\end{equation}
and that $\sup\{[y, o]: \|y\|=1\} = 0 = \| o \|$.

Let $f_x\in\X^\ast$. Then $\sup\{f_x(y): \|y\|=1\}=\sup\{ [y, x]: \|y\|=1\}$. Since $[y, x]^2\leq [y, y]\cdot [x, x]=\|x\|^2$
for all $y$ with $\|y\|=1$, we obtain
\begin{equation}\label{2}
\sup\{f_x(y): \|y\|=1\}\leq \|x\|.
\end{equation}
On the other hand, if $x \neq o$, then $\sup\{[y, x]: \|y\|=1\}\geq \left[ \frac{1}{\|x\|}x, x \right]=\frac{1}{\|x\|}[x, x]=\|x\|$, which, together with (\ref{2}), yields $\sup\{f_x(y): \|y\|=1\}=\|x\|$.
\end{proof}

\begin{pro}\label{pro3}
For the map $F$ and any $x, y\in\X$, $\lambda, \mu\in\Re$, we have
\[
\|F(\lambda x+\mu y)\|^\ast\leq |\lambda|\|F x\|^\ast+|\mu|\|F y\|^\ast.
\]
\end{pro}

\begin{proof} The definition of $F$ implies
\[
\|F(\lambda x+\mu y)\|^\ast=\|f_{\lambda x+\mu y}\|^\ast=\|\lambda x+\mu y\|\leq \|\lambda
x\|+\|\mu y\|
=|\lambda| \|F x\|^\ast+|\mu|\|F y\|^\ast.
\]
\end{proof}

\section{Antinorms}\label{sec:antinorms}

In this and the next section, we assume that $\X$ is even dimensional. Our main goal is to generalize the notion of antinorm for even dimensional normed spaces, defined in \cite{Ma-Sw2} for normed planes, and examine which of their properties remain true.

Let $\langle \cdot, \cdot\rangle$ be a (nondegenerate) bilinear symplectic form on $\X$; that is, a bilinear form satisfying
$\langle x, y\rangle=-\langle y, x\rangle$ for all $x,y \in \X$, and the property that $\langle x, y\rangle = 0$ for all $y \in \X$ yields that $x=o$.
Then the vector space $\X$ and its dual space $\X^\ast$ can be identified via
\begin{equation}\label{7}
G:\begin{array}{l} \X\rightarrow \X^\ast\\ x \mapsto g_x\end{array},\;\;\; \mathrm{where }\;\;\;g_x(y):=\langle y, x\rangle;
\end{equation}
see \cite[$\S$ 2.3]{Ma-Sw2}. It is easy to see that $G$ is an isomorphism.

We note that if $\dim \X = 2$, then every symplectic form $\langle x, y\rangle$ is a scalar multiple of the $2 \times 2$ determinant, or geometrically, up to multiplication by a constant, is the signed area of the parallelogram with vertices $o, x, x+y, y$.
On the other hand, it is well-known (cf. \cite{Silva} or \cite{arnold}) that for spaces of dimension greater than two, there are many (even though symplectically isomorphic) symplectic forms which are not scalar multiples of one another.

From now on we fix a symplectic form on $\X$.

\begin{ddd}\label{defn:antinorm}
The \emph{antinorm} of $\NS$, with respect to the symplectic form $\langle \cdot, \cdot \rangle$, is
defined, for all $x \in \X$ as
\begin{equation}\label{5}
\|x\|_a:=\|G x \|^\ast=\|g_x\|^\ast=\sup\{\langle y, x\rangle: \|y\|=1\}.
\end{equation}
For this norm, we denote the map defined in Remark~\ref{rem1} by $F_a$, and the unit ball/sphere of the antinorm by $\B_a$ and $\Sph_a$, respectively.
\end{ddd}

We note that, as it can be simply checked, the antinorm is indeed a norm defined on $\X$. Nevertheless, unlike in the plane, the antinorm relies very much on the symplectic form, i.e.  different forms yield different antinorms.

\begin{ex}
Let $\| \cdot \|$ be the $\ell_\infty$-norm on $\Re^{2n}$, and for $x=(x_1,\ldots,x_{2n})$ and $y=(y_1, \ldots,y_{2n})$, let
$\langle x,y \rangle = \sum_{i=1}^n x_i y_{i+n}-y_i x_{i+n}$. Then
$\|x\|_a = \sup\{\langle y, x\rangle:|y_1|, \ldots |y_{2n} | \leq 1 \} = \sum_{i=1}^{2n} |y_i|$, and thus, the antinorm of $x$ is its $\ell_1$-norm.
\end{ex}

The following theorem was proven in \cite{Ma-Sw2} for normed planes.
To formulate it, for any $\phi \in \X^\ast$, we write $x \perp_{\|.\|} \phi$, if $|\phi(x)| = \|\phi\|^\ast$;
that is, if the supporting hyperplane of $\|x\| \B$ at $x$ is a level surface of $\phi$.

\begin{thm}\label{thm:doubleantinorm}
Let $\| \cdot \|_{a,a}$ denote the antinorm of $(\X,\|\cdot\|_a)$ with respect to the symplectic form $\langle \cdot,\cdot, \rangle$,
where $\|\cdot \|_a$ is defined with respect to the same form.
Then, for any $x \in \X$, we have $\| x \|_{a,a} = \| x \|$.
Furthermore $x \perp_{\|.\|} G y$ if, and only if, $y \perp_{\|.\|_a} G x$.
\end{thm}

\begin{proof}
By definition, \[ \|x\|_{a,a}= \sup \{ \langle y,x \rangle : \|y\|_a = 1\}. \] Observe that (\ref{5}) yields that
$|\langle x,y \rangle| \leq \|x\| \cdot \|y\|_a$ for any $x,y \in \X$. Thus, it follows that $\|x\|_{a,a} \leq \|x\|$.
On the other hand, let $y = G^{-1} F x$. By definition, $F x=f_x$ is the linear functional with the property that $f_x(x) = \|x\|$ and
for any $z \in \B$ we have $|f_x(z)| \leq 1$, which yields that $\|f_x\|^\ast = 1$. Thus, if we set $g_y = G y$, then $f_x(z) = g_y(z)= \langle
z,y \rangle \leq 1$ for any $z \in \B$, and $g_y\left( \frac{x}{\|x\|} \right) = 1$, implying that $\|y\|_a = 1$ and $\langle x,y \rangle = \|x\|$.
Hence, by definition, $\|x\|_{a,a} \geq \|x\|$, and the first statement follows.

Now, consider some $x,y \in \X$, and assume that $|\langle x,y \rangle| = \|x\| \cdot \|y\|_a$. By the definition of antinorm, this is
equivalent to saying that the function $|\langle .,y \rangle|$ is maximized on $\|x\| \B$ at $x$. In other words, the supporting hyperplane of
$\|x\| \B$ at $x$ is a level surface of the linear functional $G y=\langle \cdot,y \rangle$. On the other hand, since $\|x\|_{a,a} = \|x\|$ and
$|\langle .,. \rangle|$ is symmetric, we have that $G x=|\langle .,x \rangle|$ is maximized on $\|y\|_a \B_a$ at $y$. Thus, we have the
following. \[ x \perp_{\|.\|} G y \quad \Longleftrightarrow \quad |\langle x,y \rangle| = \|x\| \cdot \|y\|_a \quad \Longleftrightarrow \quad y
\perp_{\|.\|_a} G x. \]
\end{proof}

The  normality relation defined at the beginning of Section~\ref{sec:prelim} is not symmetric.
Nevertheless, it was shown in \cite[$\S$ 3]{Ma-Sw2} that for any normed plane $\NS$, for any $x,y \in \X$, $x$ is normal to $y$ with respect to $\| \cdot \|$
if, and only if, $y$ is normal to $x$ with respect to $\| \cdot \|_a$, which we denote $y\dashv_a x$.
This property cannot be generalized for higher dimensions, in a strong sense, as was shown in \cite{Gru2}. Here we give a shorter proof.

\begin{thm}\label{thm:normalisnottransversal}
Let $\| \cdot \|_1$ and $\| \cdot \|_2$ be two norms defined on the real linear space $\X$, where $\dim \X > 2$.
For $i=1,2$, let $\B_i$ and $\dashv_i$ denote the unit ball and the normality relation of the norm $\| \cdot \|_i$.
Then the relations $x \dashv_1 y$ and $y \dashv_2 x$ are equivalent for all $x,y \in \X$ if, and only if $\B_1$ and $\B_2$ are homothetic ellipsoids.
\end{thm}

\begin{proof}
Clearly, if $\B_1$ and $\B_2$ are homothetic ellipsoids, then $x \dashv_1 y$ and $y \dashv_2 x$ are the same relation.

Observe that the condition $x \dashv_1 y$ geometrically means that $y$ is parallel to the supporting hyperplane of $\| x \|_1 \B_1$ at $x$.
In particular, it follows that the set $\{ y \in \X : x \dashv_1 y \}$ is a hyperplane for every $x \neq o$.
On the other hand, the set $\{ y \in \X : y \dashv_2 x \}$ is the union of the \emph{shadow boundary} of $\lambda \B_2$, $\lambda > 0$,
in the direction of $x$ (for the definition of shadow boundary, cf., e.g. \cite{GHo}).
Thus, if $x \dashv_1 y$ and $y \dashv_2 x$ are equivalent, then for any direction, the shadow boundary of $\B_2$ lies in a hyperplane.
By a result of Blaschke (cf. Theorem 10.2.3 of \cite{Sch}), this implies that $\B_2$ is an ellipsoid, and hence, $x \dashv_2 y$ and $y \dashv_2 x$ are the same relation.
We obtain similarly that $\B_1$ is an ellipsoid, which yields that $x \dashv_1 y$ and $y \dashv_1 x$ are the same relation.
Since it follows that $x \dashv_1 y$ and $x \dashv_2 y$ are the same relation as well, it is easy to see that $\B_1$ and $\B_2$ are homothetic.
\end{proof}

In light of Theorem~\ref{thm:doubleantinorm}, it is reasonable to ask if $\|\cdot \|$ and $\| \cdot\|_a$ can be proportional, or equivalently, equal for some non-Euclidean norm.
For normed planes this question was answered by Busemann \cite{Bus} (cf. also \cite{Ma-Sw2}), who proved that this happens exactly for Radon norms;
i.e. for $2$-dimensional norms in which the normality relation is symmetric.
Whereas for dimensions $n > 2$, normality is symmetric only in Euclidean spaces, Theorem~\ref{pro34} shows that the answer to our question is not so straightforward.

Before stating it, let us recall that a \emph{polar decomposition} of a symplectic form $\langle \cdot, \cdot \rangle$ on $\X$, where $\dim \X = 2n$
is a basis $\{ e_1,e_2,\ldots, e_{2n} \}$ such that $\langle e_i , e_j \rangle = 0$ if $|i-j| \neq n$, and for $i=1,2,\ldots, n$,
$1=\langle e_i,e_{i+n} \rangle = - \langle e_{i+n},e_i \rangle$.
Clearly, in this case, for any $u = \sum_{i=1}^{2n} x_i e_i$ and $v = \sum_{i=1}^{2n} y_i e_i$, their product can be written as
\[
\langle u,v \rangle = \sum_{i=1}^n x_iy_{i+n} - \sum_{i=1}^n y_ix_{i+n} .
\]
Set $U = \lin \{e_1,\ldots, e_n\}$ and $V = \lin \{e_{n+1},\ldots,e_{2n}\}$.
It is known \cite{arnold} that $U$ and $V$ are \emph{Lagrangian subspaces} of $\X$ (i.e. they are their own orthogonal complements), and, furthermore, the polar decompositions of $\X$ can be identified with pairs of transversal Lagrangian subspaces $U$ and $V$ of $\X$.
Thus, for brevity, we may call $\{ U, V \}$ a polar decomposition of $\langle \cdot, \cdot \rangle$ on $\X$.

Now, we say that $\{U,V\}$ is a \emph{Euclidean decomposition} of the norm $\| \cdot \|$, if $[\cdot,\cdot]$ is the direct sum of its restrictions to $U$ and $V$; or in other words, if for any $u \in U$ and $v \in V$, we have $\| u+v \| = \sqrt{\|u\|^2_U + \|v\|^2_V}$.
Geometrically, this condition is equivalent to the requirement that for any $u \in U$ and $v \in V$, the intersection of $\B$ with $\lin \{ u,v \}$ is an ellipse, where $u$ and $v$ belong to a pair of conjugate diameters.
We note that the semi-inner product defined in this way is also a semi-inner product \cite{Ho}, and that this property (and its geometric variant) appeared also in \cite{La}.

\begin{thm}\label{pro34}
Assume that $\{ U, V \}$ is a Euclidean decomposition of $\NS$, where $\dim \X = 2n$.
Let $\langle \cdot, \cdot \rangle $ be a symplectic form on $\X$ with a polar decomposition $\{U,V\}$.
Then the following are equivalent.
\begin{itemize}
\item[(i)] The antinorm $\| \cdot \|_a$, with respect to $\langle \cdot, \cdot \rangle$, is equal to $\| \cdot \|$.
\item[(ii)] We have
\begin{equation}\label{eq:polardecomp}
\B \cap U = \{ x \in U : | \langle x, y \rangle | \leq 1 \hbox{ for every } y \in \B \cap V \},
\end{equation}
\[
\B \cap V = \{ y \in V : | \langle x, y \rangle | \leq 1 \hbox{ for every } x \in \B \cap U \}.
\]
\end{itemize}
\end{thm}

Note that if we imagine $U$ and $V$ as orthogonal subspaces, then (ii) states that, identifying $U$ and $V$ via a symplectic basis,
$U \cap \B$ and $V \cap \B$ are polars of each other.

\begin{proof}
We set $U = \lin \{ e_1,e_2,\ldots,e_n \}$, $V=\lin \{ e_{n+1}, \ldots, e_{2n} \}$, where $\langle e_i , e_j \rangle = 0$ if $|i-j| \neq n$,
and for $i=1,2,\ldots, n$, $\langle e_i,e_{i+n} \rangle = 1$.
For simplicity, we imagine this basis as the standard orthonormal basis of an underlying Euclidean space.
Let $K = U \cap \Sph$ and $L = V \cap \Sph$.
By straightforward computation, from the definition of Euclidean decomposition of $\NS$, we obtain that
\[
\Sph = \left\{ u\cos \phi + v \sin \phi : u \in K, v \in L, \hbox{ and } \phi \in [0,2\pi] \right\}.
\]
(We note that the opposite direction also holds, for the idea of the proof see \cite[Lemma 2]{La}.)

First we prove (ii) $\Rightarrow$ (i).

To prove (i) observe that, by the definition of antinorm, we have $\|x\|_a = \sup \{ \langle x,y \rangle : \|y\| = 1 \}$. By homogeneity, it suffices
to show that for our norm, $\|x\| = 1$ yields $\sup \{ \langle x,y \rangle : \|y\| = 1 \} = 1$. In other words, we
need to show that $\langle \Sph,\Sph \rangle = [-1,1]$, and for every $x \in \Sph$, there is some $y \in \Sph$ satisfying $\langle x,y \rangle = 1$.

Consider some $x,y \in \Sph$. Then $x=(x_1 \cos \alpha,\ldots, x_n \cos \alpha, x_{n+1} \sin \alpha,\ldots, x_{2n} \sin \alpha)$ and $y = (y_1
\cos \beta,\ldots , y_n \cos \beta, y_{n+1} \sin \beta,\ldots , y_{2n} \sin \beta)$, where $(x_1,\ldots, x_n), (y_1,\ldots y_n) \in K$,
$(x_{n+1},\ldots x_{2n}), (y_{n+1},\ldots ,y_{2n}) \in L$, and, without loss of generality, $0 \leq \alpha, \beta \leq \frac{\pi}{2}$.
An elementary computation yields that
\[
\langle x,y \rangle = (x_1 y_{n+1} +\ldots +x_n y_{2n}) \cos \alpha \sin \beta - (x_{n+1}y_1+ \ldots +
x_{2n}y_n)\sin \alpha \cos \beta .
\]
By the definitions of $K$ and $L$, we have $| x_1 y_{n+1} +\ldots +x_n y_{2n}| \leq 1$ and
$|x_{n+1}y_1+ \ldots + x_{2n}y_n| \leq 1$. Thus,
\[
|\langle x,y \rangle| \leq \cos \alpha \sin \beta + \sin \alpha \cos \beta = \sin(\alpha+\beta) \leq 1.
\]

On the other hand, consider any $x \in \Sph$.
Then, using the notations of the previous paragraph, we have $(x_1,\ldots,x_n,0,\ldots,0) \in K$ and
$(0,\ldots,0,x_{n+1},\ldots,x_{2n}) \in L$. Thus, by the condition in (ii),
there are some $(y_1,\ldots,y_n,0,\ldots,0) \in K$ and $(0,\ldots,0,y_{n+1},\ldots,y_{2n}) \in L$ satisfying
$\sum_{i=1}^n x_i y_{n+i} = -\sum_{i=1}^n y_i x_{n+i} = 1$.
Now, setting
\[
y=\left( y_1 \cos \left( \frac{\pi}{2} - \alpha \right), \ldots, y_n \cos \left( \frac{\pi}{2} - \alpha \right),y_1 \sin \left( \frac{\pi}{2} - \alpha \right), \ldots, y_n \sin \left( \frac{\pi}{2} - \alpha \right) \right),
\]
we have $\langle x,y \rangle = \sin \frac{\pi}{2} = 1$.

Finally, we prove (i) $\Rightarrow$ (ii).

Assume that $\|\cdot\|=\|\cdot \|_a$ holds, and let $x = (x_1,\ldots,x_{2n})\in \Sph$.
Then, we have
\[
1 = \sup \{ \langle x,y \rangle : y \in \Sph \} = \sup \left\{ \sum_{i=1}^n x_i y_{i+n} : (y_1,\ldots,y_{2n}) \in \Sph \right\} .
\]
Observe that, by the definition of Euclidean decomposition, the orthogonal projection of $\Sph$ onto $V$ is $V \cap \B = \conv L$. Thus,
we have
\[
1 = \sup \left\{ \sum_{i=1}^n x_i y_{i+n} : (0,\ldots,0,y_{n+1},\ldots,y_{2n}) \in L \right\} = \sup \left\{ \langle x,y \rangle : y \in L \right\}.
\]
This yields the first equality in (ii), which readily implies the second inequality as well.
\end{proof}

\begin{cor}\label{cor4}
There are infinitely many non-Euclidean norms coinciding with their antinorms with respect to some symplectic form.
\end{cor}

\begin{cor}\label{cor5}
Assume that in Theorem~\ref{pro34}, $\B \cap U$ and $\B \cap V$ are ellipsoids, and that the polar decomposition defined by $U, V$ is an orthogonal basis of an underlying Euclidean space. Let $\B^\ast$ denote the (Euclidean) polar of $\B$ in this space (cf. (\ref{eq:polarinX})), defining the norm $\| \cdot \|_{\ast}$ and antinorm $\| \cdot \|_{\ast,a}$.
Then we have
\[
(\ref{eq:polardecomp}) \Leftrightarrow \|\cdot\|=\|\cdot\|_a \Leftrightarrow \|\cdot\|_{\ast}=\|\cdot\|_{\ast,a}.
\]
\end{cor}

\begin{proof}
Clearly, it suffices to prove the second equivalence. Let the polar basis of $\langle \cdot, \cdot \rangle$ be $\{ e_1,\ldots,e_{2n} \}$, and assume that
$\|\cdot\|=\|\cdot\|_a$. This, by the definition of polar decomposition, yields that
the semi-axes of $\B$ are $a_1, \ldots, a_n, \frac{1}{a_1}, \ldots, \frac{1}{a_n}$, in the directions of the corresponding basis vectors, respectively.
On the other hand, $\B^\ast$ is also an ellipsoid, with semi-axes $\frac{1}{a_1}, \ldots, \frac{1}{a_n}, a_1, \ldots, a_n$ in the same directions, respectively,
and thus, it satisfies the conditions in (\ref{eq:polardecomp}). The opposite direction follows from $(\B^\ast)^\ast = \B$.
\end{proof}

Note that if $\|\cdot\|_1$ and $\|\cdot\|_2$ are norms on $\X$, then for any $1 \leq p \leq \infty$,
$\| \cdot \| = \left(\|\cdot\|_1^p+ \|\cdot\|_2^p\right)^{\frac{1}{p}}$ is a norm as well.
In the following, we examine the relation between antinorm and this operation.
We remark that for $p=1$, using the identities between the support and the gauge/radial functions (cf. \cite{gardner}),
we have that  the unit ball $\B$ of $\| \cdot \|$ is the convex body $(\B_1^{\ast}+\B_2^{\ast})^\ast$, where $\B_1$ and $\B_2$ are the unit balls of $\| \cdot \|_1$
and $\| \cdot \|_2$, respectively.

\begin{pro}\label{pro36}
Let $(\X,\langle \cdot,\cdot \rangle)$ be a symplectic vector space. For $i=1,2$, let $\| \cdot \|_i$ be a norm on $\X$ with unit ball $\B_i$, and with antinorm
$\| \cdot \|_{i,a}$. Let $1 \leq p \leq \infty$, and let $\| \cdot \| =  \left(\|\cdot\|_1^p+ \|\cdot\|_2^p\right)^{\frac{1}{p}}$.
Then, for every $x \in \X \setminus \{ o \}$, we have
\[
\| x \|_a \leq \min \{ \| x\|_{1,a} , \| x\|_{2,a} \} \leq \left(\| x\|_{1,a}^p+ \| x\|_{2,a}^p\right)^{\frac{1}{p}},
\]
with equality in the first inequality if, and only if, $p = \infty$, and in the second one if, and only if $p = \infty$ and $\|x\|_{1,a}=\|x\|_{2,a}$.
\end{pro}

\begin{proof}
By the definition of antinorm, for $i=1,2$, we have
\[
\| x \|_ a = \sup \left\{ \langle x,y \rangle :  \left(\|y\|_1^p+ \|y\|_2^p\right)^{\frac{1}{p}} \leq 1 \right\} \leq \sup \left\{ \langle x,y \rangle :  \|y\|_i \leq 1 \right\} \leq \| x \|_{i,a} .
\]
From this, the assertion readily follows.
\end{proof}

\section{The normality map}\label{sec:normalitymaps}

Let $\X$ be even dimensional, $[\cdot,\cdot]$ induced by a strictly convex, smooth norm,
and let $\langle \cdot, \cdot \rangle$ be a symplectic form on $\X$.
Recall the maps $F$ from Remark~\ref{rem1} defined on $\NS$, and $G$, defined in (\ref{7}) for $(\X, \langle \cdot,\cdot \rangle )$.
The main concept of this section is the following.

\begin{ddd}\label{defn:normalitymap}
The product $J = G^{-1} F$, $J : \X \rightarrow \X$ is called the \emph{normality map} of $\NS$, with respect to $\langle \cdot, \cdot \rangle$.
\end{ddd}

We remark that the normality map $J$ also appears in \cite[p. 308]{Gu} as $T$.

Using the definitions of $F$ and $G$, $J$ can be interpreted as follows: for every linear functional $f \in \X^\ast$ there are unique vectors $x,x' \in \X$ such that $f(\cdot) = [\cdot, x] = \langle \cdot, x' \rangle$. In this case $J(x) = x'$.

\begin{ex}
It is easy to check that if $\{ U,V \}$ is a polar decomposition of $\NS$, and $[\cdot, \cdot]$ is an inner product on $\X$ such that $U$ and $V$ are orthogonal, then $J(x)$ is the reflection of $x$ either about $U$ or about $V$.
\end{ex}

\begin{rem}\label{rem:notlinear}
As $G$ is linear, the linearity of $J$ implies that $F$ is linear, and thus, $J$ is linear if, and only if $(\X,\| \cdot \|)$ is a Hilbert space.
\end{rem}

In light of Remark~\ref{rem:notlinear}, we would like to emphasize that by \emph{isometry}, we mean a (not necessarily linear) transformation $Z : \X \to \X$ with the property that for every $x \in \X$, $\| Z(x)\| = \| x \|$.

Before stating our first result in this section, we set $J_a=G^{-1} F_a$, and recall that $\B_a$ and $\Sph_a$ denote, respectively,
the unit ball and the unit sphere of the antinorm.

\begin{thm}\label{pro33}
For any $x, y\in \X$ and any $\lambda\in \Re$, we have
\begin{enumerate}
\item[(i)] $\|x\|=\|J x\|_a$ and $\|x\|_a=\|J_a x \|$;
\item[(ii)]  $J \Sph =\Sph_a$ and $J \B = \B_a$;
\item[(iii)] $[x, y]=\langle x, J y\rangle$ and $[x, y]_a=\langle x, J_a y \rangle$,  where $[\cdot, \cdot]_a$ is the semi-inner product induced by $\|\cdot\|_a$;
\item[(iv)] $x\dashv J x$ and  $x\dashv_a J_a x$;
\item[(v)] $[J x , y]=-[J y, x]$;
\item[(vi)]  $[J x , y]_a=-[J_a y, x]$;
\item[(vii)]  $J(\lambda x)=\lambda J x $;
\item[(viii)] $J_a J=J J_a=-I$, where $I$ denotes the identity map of $\X$;
\item[(ix)] $[x, y]=[J y , J_a  x ]$ and $[x, J_a y ]=-[y, J_a x]$.
\end{enumerate}
\end{thm}

\begin{proof}

Let $x\overset{F}{\longmapsto} f_x \overset{G^{-1}}{\longmapsto} J x$ and $x\overset{F_a}{\longmapsto} f_x^a
\overset{G^{-1}}{\longmapsto} J_a x$. Since $G J x=f_x$ and $\|x\|_{a,a}=\|x\|$, we have
\begin{equation}\label{6}
\|J x \|_a=\|G J x\|^\ast=\|f_x\|^\ast=\|x\|,
\end{equation}
by (\ref{4}) and (\ref{5}). The second equality in (i) is implied by $G J_a x=f^a_x$.

The equality (\ref{6}) yields (ii). According to the definition of $G$ in (\ref{7}), we have $G J x =\langle \cdot, J x  \rangle$.
On the other hand, $G J x =G G^{-1}F x=F x=[\cdot, x]$.  The same  holds also for $J_a$, and thus we obtain (iii).
Setting $x=J y$ in the first equality in (iii) and  $x=J_a y$ in the second one yields (iv).

By (iii) and the skew-symmetry of $\langle \cdot, \cdot\rangle$ it follows that
\[
[J x, y]=\langle J x , J y \rangle=-\langle J y , J x \rangle=-[J y, x],
\]
which proves (v).
By a similar argument, we may obtain (vi).
The homogeneity  of $[\cdot, \cdot]$ and $\langle \cdot, \cdot\rangle$ yields  (vii).

By (i) and $\|x\|_{a,a} = \|x\|$, we have $\|J_a J x \|= \|x\|$ for every $x \in \X$. Thus, $J_a J \B = \B$, which yields that $J_a J$ is contained in
the symmetry group of $\B$. Since $\B$ is $o$-symmetric, this group contains $I$ and $-I$, and possibly some other transformations.
First, consider the case that the only symmetries of $\B$ are $I$ and $-I$, which implies that $J_a J = I$ or $J_a J = -I$.
Furthermore, if $J_a J = I$, then, applying (vi) with $y=J x$ yields $[J x,J x]_a = - [x,x]$. Since
for every $x \in \X$, we have $[J x,J x]_a \geq 0$, and $[x,x] \geq 0$ with equality only for $x=o$, we have reached a contradiction, implying that $J_a J = -I$
in this case. If $\B$ has symmetries different from $I$ and $-I$, then we may approach $\B$ with a sequence of $o$-symmetric convex bodies which have no other symmetries, and apply a continuity argument. This proves the first part of (viii), whereas the second part follows from the same argument.

Finally, from (v) it follows that $[J J_a x, y]=-[J y , J_a x ]$, which, together with (viii), yields the first relation of (ix), implying
$[x, J_a y]=[J J_a y, J_a x]=-[y, J_a x]$ as well.
\end{proof}

\begin{rem}\label{rem5}
If $\NS$ is the Euclidean plane, then $J: \X \longrightarrow \X$ is simply the rotation about the origin by $\frac{\pi}{2}$.
Furthermore, if $\NS$ is two-dimensional,  then $J \Sph=\Sph_a$ is the \emph{isoperimetrix} of $\NS$ (cf. \cite{Th}).
\end{rem}

For our next theorem, we need some preparation.
First, recall the so-called `linear Darboux Theorem' (cf. \cite{arnold}) that states that
any two symplectic spaces $(\X_1, \langle \cdot, \cdot \rangle_1)$ and  $(\X_2, \langle \cdot, \cdot \rangle_2)$ of the same dimension are \emph{symplectically isomorphic}; that is, there is a linear isomorphism $L : \X_1 \to \X_2$ satisfying $\langle x,y\rangle_1=\langle L x,L y\rangle_2$ for all $x,y \in \X_1$.
Such a map is called a \emph{symplectic isomorphism} between the two spaces.
Furthermore, observe that if $(\X, \langle \cdot, \cdot \rangle )$ is a symplectic space and $L : \X \longrightarrow \X$ is a linear transformation,
then there is a unique linear transformation $L'$ satisfying $\langle L x,  y \rangle = \langle  x, L' y \rangle$, or equivalently, $\langle x, L y \rangle = \langle L' x, y \rangle$, which we call the \emph{left adjoint}
of $L$.

\begin{rem}
If, in a polar decomposition of $\X$, where $\dim \X = 2n$, the matrix of $L$ is $\left[ \begin{array}{cc} A_1 & A_2 \\ A_3 & A_4 \end{array} \right]$,
where each block is an $n \times n$ matrix, then the matrix of its left adjoint is $\left[ \begin{array}{cc} A_4^T & -A_2^T \\ -A_3^T & A_1^T \end{array} \right]$.
\end{rem}

\begin{thm}\label{pro35}
Let $\langle \cdot,\cdot\rangle_1$ and $\langle \cdot,\cdot\rangle_2$ be two symplectic forms on $\NS$, and let
$L : (\X, \langle \cdot,\cdot\rangle_1) \longrightarrow (\X, \langle \cdot,\cdot\rangle_2)$ be a symplectic isomorphism.
For $i=1,2$ and any linear transformation $A$ on $\X$, let $A_i^\ast$, $J_i$ and $\| \cdot \|_{i,a}$ denote the (left or right) adjoint of $A$, the normality map and the antinorm, respectively, with respect to $\langle \cdot,\cdot\rangle_i$. Then the following are equivalent.
\begin{enumerate}
\item[(i)] $J_1^{-1}L_2^\ast LJ_1$ is an isometry of $\NS$.
\item[(ii)] $J_2^{-1} (L^{-1})_1^\ast L^{-1} J_2$ is an isometry of $\NS$.
\item[(iii)] For any $x \in \X$, we have $\|x\|_{1,a}=\|x\|_{2,a}$.
\end{enumerate}
\end{thm}

We note that since the unit ball $\B$ of $\NS$ is $o$-symmetric, it makes no difference if, in Theorem~\ref{pro35} we mean right or left adjoint.

\begin{proof}
By symmetry, it suffices to show that (i) and (iii) are equivalent.
First, assume that (i) holds.
Then, for any $x \in \X$,
\[
\|x\|_{1,a}=\sup\left\{ \langle y,x\rangle_1: y\in \Sph \right\}=\sup\left\{ \langle
L y,L x\rangle_2: y\in \Sph \right\}=
\]
\[
=\sup\left\{ \langle y,L_2^\ast L  x \rangle_2: y\in \Sph\right\}=\|L_2^\ast  L x\|_{2,a}.
\]
Hence, $L_2^\ast  L \Sph_{1,a} =\Sph_{2,a}$, where, for $i=1,2$, $\Sph_{i,a}$ is the unit sphere of the norm $\| \cdot \|_{i,a}$.
Now, from (ii) of Theorem \ref{pro33} it follows that
\[
\Sph_{2,a}=L_2^\ast  L \Sph_{1,a}=J_1 J_1^{-1}L_2^\ast LJ_1 \Sph=J_1 \Sph=\Sph_{1,a},
\]
implying (iii).

Conversely, (iii) of Theorem~\ref{pro33} yields
\[
\langle y, J_2 x \rangle_2=[y,x]=\langle y, J_1 x\rangle_1=\langle L y, L J_1 x\rangle_2=\langle y, L_2^\ast L J_1 x\rangle_2
\]
holds for all $x,y\in \X$. Since $\langle \cdot, \cdot \rangle_2$ is nondegenerate and bilinear, from this
$J_2 x = L_2^\ast LJ_1 x$ follows for all $x \in \X$.
Now, using (ii) of Theorem \ref{pro33} and (iii), we obtain
\[
J_1 \Sph=\Sph_{1,a}=\Sph_{2,a}=J_2 \Sph = L_2^\ast L J_1 \Sph,
\]
which readily yields (i).
\end{proof}

In light of Theorem~\ref{pro35}, we may introduce a more refined classification system on symplectic forms than standard symplectic isomorphism.
Theorem~\ref{pro35} shows also that antinorms and the isometries of a normed space are related.

\begin{ddd}\label{equivalence}
Let $\langle \cdot,\cdot\rangle_1$ and $\langle \cdot,\cdot\rangle_2$ be two symplectic forms on $\NS$.
For $i=1,2$, let $\| \cdot \|_{i,a}$ denote the antinorm with respect to $\langle \cdot,\cdot\rangle_i$.
We say that the two symplectic forms are \emph{equivalent with respect to the norm $\| \cdot \|$}, if
for all $x \in \X$, we have $\| x \|_{1,a} = \| x \|_{2,a}$.
\end{ddd}

\section{The concept of semi-polarity}\label{sec:semipolar}

The concept of \emph{polarity} (or \emph{polar duality}) is a very important tool in several areas of convexity.
From a functional analytic point of view, for a real vector space $\X$, the \emph{polar} of a set $X \subset \X$ is defined
as the subset $\{f\in \X^\ast: f(x)\leq 1\;\text{for all}\;x\in X\}$ of the dual space.
Nevertheless, in geometry, this subset is usually identified with the subset
\begin{equation}\label{eq:polarinX}
X^\ast=\{y\in \X: [x, y]_{\Eu}\leq 1\;\text{for all}\;x\in X\}
\end{equation}
of $\X$, induced by an inner product $[\cdot, \cdot]_{\Eu}$ of $\X$.
This subset $X^\ast$ is also called the polar, and we use this definition in our paper.
We note that the identification in (\ref{eq:polarinX}) assumes an inner product structure on the space, as there is no canonical isomorphism
between $\X$ and $\X^\ast$.
The following theorem summarizes some of the important properties of polar sets; see,
e.g.  \cite[$\S$ 1.6 and Remark 1.7.7]{Sch}, \cite[$\S$ 2.8]{Web}, \cite[$\S$ 3]{Bo-Ma-So}, and \cite[$\S$ 4.1, p. 56]{Gru}.
Recall from Section~\ref{sec:prelim} that $\XX$ denotes the family of convex bodies in $\X$, with the origin $o$ as an interior point,
and in a normed space with unit ball $\B$, the support function of $K \in \XX$ in the direction $u \neq o$ is $h_B(K,u) = \sup \{ [x,u] : x \in K \}$, and the gauge function of $K$ is $g(K,x) = \min \{ \lambda \geq 0 : x \in \lambda K \}$, $x \in \X$. Furthermore, $\B_{\Eu}$ is the Euclidean unit ball of $\X$.

\begin{thm}\label{polarity}
Let $M, N \subset \X$, and $\lambda \neq 0$. Then
\begin{enumerate}
\item[(i)] $M\subseteq N$ implies $N^\ast\subseteq M^\ast$;
\item[(ii)] $(M\cup N)^\ast=M^\ast\cap N^\ast$;
\item[(iii)] $(\lambda M)^\ast=(1/\lambda)M^\ast$;
\item[(iv)] $\B_{\Eu}^\ast=\B_{\Eu}$.
\end{enumerate}
If $M\in \XX$, then
\begin{enumerate}
\item[(v)] $M^{\ast\ast}=M$,
\item[(vi)] $\displaystyle g(M^\ast, x)=h(M, x)$ and   $h(M^\ast, x)=g(M, x)$.
\end{enumerate}
If, in addition, $M$ is centered at $o$, then
\begin{enumerate}
\item[(vii)] $h(M, x)=\|x\|_{M^\ast}$ and $h(M^\ast, x)=\|x\|_M$ for $x\in X$, where $\|\cdot\|_N$ is the norm induced by
the $o$-symmetric convex body $N$.
\end{enumerate}
\end{thm}

To generalize this notion, instead of an inner product, we use the semi-inner product of $\NS$ to identify elements of $\X$ and $\X^\ast$.
Unless we specifically state, in this section we do not restrict our investigation to even dimensional spaces but consider only strictly convex, smooth norms.

\begin{ddd}\label{defn:semipolar}
Let $\NS$ be a normed space with unit ball $\B$ and semi-inner product $[\cdot, \cdot]$ and let $m \in \X$.
Then the \emph{left/right semi-polar of $m$} is
\begin{equation}\label{eq:semipolarm}
m_\circ = \{ x \in \X : [m,x] \leq 1 \}\quad \hbox{and} \quad m^\circ = \{ x \in \X : [x,m] \leq 1 \} ,
\end{equation}
respectively.
If $M \subset \X$, then the \emph{left/right semi-polar of $M$} is
\begin{equation}\label{eq:semipolarMleft}
M_\circ = \bigcap_{m \in M} m_\circ =\{ x \in \X : [m,x] \leq 1 \hbox{ for all } m \in M\}
\end{equation}
\begin{equation}\label{eq:semipolarMright}
M^\circ = \bigcap_{m \in M} m^\circ =\{ x \in \X : [x,m] \leq 1 \hbox{ for all } m \in M\} ,
\end{equation}
respectively.
\end{ddd}

We note that this definition implies $o_\circ=o^\circ = \X$. Observe that $f_x = [\cdot,x]$ is a linear functional on $\X$, but $[x,\cdot]$ is not necessarily so.
Thus, if $M \in \XX$, then $M^\circ \in \XX$ as well, but $M_\circ$ is not necessarily convex.
On the other hand, if $[\cdot,\cdot]$ is symmetric (e.g. $\NS$ is an inner product space), then both $M_\circ$ and $M^\circ$ coincide with the usual polar of $M$ in this space.

\begin{thm}\label{pro7}
Let $M,N \subset \X$ and $\lambda \neq 0$. Then
\begin{enumerate}
\item[(i)] $M\subseteq N$ implies $N^\circ\subseteq M^\circ$ and $N_\circ \subseteq M_\circ$;
\item[(ii)] $(M\cup N)^\circ=M^\circ\cap N^\circ$ and $(M\cup N)_\circ=M_\circ\cap N_\circ$;
\item[(iii)] $(\lambda M)^\circ=(1/\lambda)M^\circ$ and $(\lambda M)_\circ=(1/\lambda)M_\circ$;
\item[(iv)] $\B^\circ=\B_\circ = \B$.
\item[(v)] If $M \in \XX$, then $(M_\circ)^\circ = M$.
\end{enumerate}
\end{thm}

\begin{proof}
Note that (i)-(iii), and the equality $\B^\circ = \B$ are straightforward consequences of Definition~\ref{defn:semipolar}.

We prove that $\B_\circ = \B$.
By definition, we have $\B_\circ = \{ y \in \X : [x,y] \leq 1 \hbox{ for any } x \in \B \}$.
Since $[x,y] \leq 1$ for any $x,y \in \B$, we clearly have $\B \subseteq \B_\circ$.
On the other hand, let $y \in \X \setminus \B$. Then $\|y\| > 1$, and we have $\left[ \frac{y}{\|y\|},y\right]= \frac{1}{\|y\|} [y,y] = \|y\| > 1$.
As $\frac{y}{\|y\|} \in \B$, it follows that $y \notin \B_\circ$.

Finally, we show (v). By definition, for any $M \subset \X$, we have $M \subseteq (M_\circ)^\circ$.
Let $x \notin M \in \XX$. Then there is a hyperplane $H$ strictly separating $M$ from $x$.
Since $H$ cannot pass through $o$, using the identification $F$ in Remark~\ref{rem1}, there is some $y \in \X$ such that
$H = \{ z \in \X : [z,y] = 1 \}$.
Now, for any $z \in M$, we have $[z,y] < 1$, implying that $y \in M_\circ$.
On the other hand, $[x,y] > 1$, which yields that $x \notin (M_\circ)^\circ$ and $(M_\circ)^\circ = M$.
\end{proof}

\begin{thm}\label{thm:innerproductcharacterization}
For any $\NS$, the following are equivalent.
\begin{enumerate}
\item[(i)] $\NS$ is an inner product space;
\item[(ii)] for any $m \in \X$, $m_\circ$ is convex;
\item[(iii)] for any $m \in \X$, $m \in (m^\circ)^\circ$;
\item[(iv)] for any $m \in \X$, $m \in (m_\circ)_\circ$.
\end{enumerate}
\end{thm}

\begin{proof}
Clearly, (i) implies (ii)-(iv).

First, we show that (ii) implies (i).
Assume that (ii) holds. Then, since the intersection of convex sets is convex, we have that $M_\circ$ is convex for any $M \subset \X$.
Let $m \neq o$ arbitrary, and $M = \{ \lambda m : \lambda \in \Re \}$. Then $x \in M_\circ$ if, and only if $[m,x]=0$, or in other words,
if $x \dashv m$. Observe that the set of these points is exactly the conic hull of the shadow boundary of $\B$, in the direction of $m$.
Since $\B$ is strictly convex, or in other words, $\Sph$ does not contain a nondegenerate segment, from the convexity of $M_\circ$
it follows that $M_\circ$ is a hyperplane, passing through the origin. Thus, similarly as in the proof of Theorem~\ref{thm:normalisnottransversal},
to finish the proof it suffices to apply the result of Blaschke (cf. Theorem 10.2.3 of \cite{Sch}), stating that in this case $\B$ is an ellipsoid.

Now we prove that (iii) yields (i).
Assume that for any $m \in \X$, $m \in (m^\circ)^\circ$. Then we have $[m,x] \leq 1$ for any $x \in m^\circ$; or in other words,
$[x,m] \leq 1$ implies $[m,x] \leq 1$, for any $x,m \in \X$. We show that from this, it follows that $[x,m]=1$ and $[m,x]=1$ are equivalent.
Indeed, assume that $[x,m]=1$ and $[m,x] < 1$ for some $x,m \in \X$. Then, by the homogeneity of the second variable, there is some $\lambda > 1$ such that $[m,\lambda x ] \leq 1$, which implies $1 \geq [\lambda x, m ] = \lambda [x,m] = \lambda > 1$; a contradiction.
Hence, we have that $[x,m]=1$ and $[m,x]=1$ are equivalent, which yields, by homogeneity, that $[x,m]=[m,x]$ for any $m,x \in \X$.
Thus $[\cdot,\cdot]$ is an inner product.

To show that (iv) yields (i), we may apply a similar argument.
\end{proof}

For even dimensional spaces, Theorem~\ref{thm1} seems to be the analogue of (v) of Theorem~\ref{polarity}.

\begin{thm}\label{thm1}
Let $\NS$ be even dimensional, and let $J$ be the normality map with respect to a symplectic form $\langle \cdot,\cdot \rangle$ on $\X$.
If $M\in \XX$, then $\conv(J M)=(J_a M^\circ )^\circ$.
\end{thm}

We note that $J M$ is not necessarily convex, even in the plane. As an example,
we can take $\|.\|$ as the $\ell_p$-norm with $p \approx \infty$, and $M$ as the unit disk of the $\ell_1$ norm.

\begin{proof}[Proof of Theorem~\ref{thm1}]
By definition, $M^\circ=\{x\in \X: [x, m]\leq 1\;\;\;\text{for all}\;\;\; m\in M\}.$ Hence, from (ix) of Theorem~\ref{pro33}, it follows that
$[J m , J_a x ]\leq 1$ holds for every $x\in M^\circ$ and every $m\in M$. Therefore $J m \in (J_a M^\circ )^\circ$, implying
$\conv (J M )\subseteq(J_a M^\circ )^\circ$.

To prove that $(J_a M^\circ )^\circ \subseteq \conv(J M)$, consider some $z\not\in \conv(J M)$.
Then there is a hyperplane $H$ strictly separating $z$ and $\conv(J M)$. Since $\conv(J M )\in \XX$, this hyperplane cannot
pass through the origin, and, using the identification $F$ of the elements of $\X$ and $\X^\ast$ in Remark~\ref{rem1},
$H = \{ x \in \X : [x,u]=1\}$ for some $u \in \X$. Then
\begin{equation}\label{16}
[z, u]>1, \;\;\;\text{and} \;\;\; [J m, u]<1\;\;\;\text{for any}\;\;\; m\in M.
\end{equation}
Now, (ix) of Theorem~\ref{pro33} implies that
\[
[J_a^{-1} u, m]=[J m, J_a J_a^{-1} u )] = [J m,u]< 1,
\]
for every $m \in M$, from which $J_a^{-1} u \in M^\circ$ and $u \in J_a M^\circ$ follows.
Thus, by (\ref{16}) we have that $z \notin (J_a M^\circ)^\circ$, which completes the proof.
\end{proof}

The next theorem shows how the gauge function of the semi-polar of a convex body is related to the normed support function of this body.

\begin{thm}\label{thm3}
Let $\NS$ be even dimensional, and let $J$ be the normality map with respect to a symplectic form $\langle \cdot,\cdot \rangle$ on $\X$.
Assume that $M, J M , J_a M^\circ\in \XX$. Then
\begin{equation}\label{18}
h_B(M^\circ, x)=g(M, x)\;\;\; \hbox{and} \;\;\;h_B(J M , x)=g(J_a M^\circ, x)
\end{equation}
for every  $x\in \X\setminus\{o\}$.
\end{thm}

\begin{proof}
First we show that the second equation in (\ref{18}) implies the first one.
Applying the second equation for $J_a M^\circ $ and using Theorem~\ref{thm1}, we obtain $h_B(J J_a M^\circ, x)=g(J_a J M, x)$,
which, by (viii) of Theorem~\ref{pro33}, is equivalent to $h_B(-M^\circ, x)=g(-M, x)$, and $h_B(M^\circ, x)=g(M, x)$.

Now, we prove the second equation.
Note that by our assumptions, $o$ is an interior point of $J_a M^\circ $.
Hence, for any $x \neq 0$, we may denote by $x_0$ the intersection point of
$\bd J_a M^\circ$ with the conic hull of $x$.  Let $s = J_a^{-1} x_0 \in M^\circ$.
Then, by (ix) of Theorem~\ref{pro33}, for every $m\in M$ we have
\[
1\geq [s, m]=[J_a^{-1} x_0, m] \Longrightarrow 1\geq [J m, J_a J_a^{-1} x_0]=[J m, x_0].
\]
Thus, we obtain that $h_B(J M, x_0)=\sup\{[J m, x_0]: m\in M\} \leq 1,$ which yields that
\begin{equation}\label{19}
h_B(J(M), x)=h_B \left( J M, \frac{\|x\|}{\|x_0\|}x_0 \right) \leq \frac{\|x\|}{\|x_0\|}=g(J_a M^\circ, x).
\end{equation}

On the other hand, let $0< \lambda< g(J_a M^\circ, x)$ be arbitrary. Then, by (vii) of Theorem~\ref{pro33} and (iii) of Theorem~\ref{pro7}, we have
\begin{equation}\label{20}
x\not \in \lambda J_a M^\circ =J_a(\lambda M^\circ)=J_a \left( \left( \frac{1}{\lambda} M \right)^\circ \right).
\end{equation}
Applying this for $y= J_a^{-1} x$, we obtain that $y\not\in \left( \frac{1}{\lambda} M \right)^\circ$, which yields
$\left[ y, \frac{1}{\lambda} m_0 \right]> 1$ for some $m_0\in M$. Hence, by (ix) of Theorem~\ref{pro33} and the homogeneity of $[\cdot,\cdot]$,
\[
\lambda < [J_a^{-1} x, m_0]=[J m_0, J_a J_a^{-1} x]=[J m_0, x],
\]
and therefore $h_B(J M, x)=\sup\{[J m, x]: m\in M\}> \lambda.$
Since $0< \lambda< g(J_a M^\circ, x)$ is arbitrary, it follows that $h_B(J M, x)\geq g(J_a M^\circ, x)$, which,
combined with (\ref{19}), proves the assertion.
\end{proof}

\begin{cor}\label{cor2}
If $M\in\XX$, then $h(M^\ast, x)=h_B(M^\circ, x)$ and $h(M^\ast, x) h(\B, x)=h(M^\circ, x)$.
\end{cor}

\begin{proof}
It follows from (vi) of Theorem~\ref{polarity} and (\ref{18}) that $h(M^\ast, x)=g(M, x)=h_B(M^\circ, x)$.
\end{proof}

The next corollary is an analogue of (vii) of Theorem~\ref{polarity}. We note that if $M$ is $o$-symmetric, then so are $J M$,
$J_a M$ and $M^\circ$.

\begin{cor}\label{cor3}
If $M, J M, J_a M^\circ \in \XX$ and $M$ is $o$-symmetric, then
\[
h_B(M^\circ, x)=\|x\|_M\;\;\;\text{and}\;\;\;h_B(J M, x)=\|x\|_{J_a M^\circ}.
\]
\end{cor}


\section{Remarks and questions}\label{sec:remarks}

\begin{rem}
One can attribute a geometric meaning to a symplectic form in any dimensions.
More specifically, if $\{e_1,\ldots,e_{2n}\}$ is a polar decomposition of the symplectic product $\langle \cdot,\cdot\rangle$, then $\langle
x,y \rangle$ is the sum of the areas of the projections onto the $n$ coordinate planes $\{e_i,e_{n+i}\}$ of the oriented parallelogram which $x$
and $y$ span.
\end{rem}

It is clear from Theorem~\ref{pro34} that if $\X$ is a Euclidean space, and $\langle \cdot, \cdot \rangle$ has a polar decomposition into an orthonormal basis of $\X$, then, with respect to this form, $\| \cdot \|_a = \| \cdot \|$. On the other hand, it is easy to see that these two norms are not even proportional
for \emph{each} symplectic form for any norm. Note that, for normed spaces, the counterpart of an orthogonal basis is a so-called \emph{Auerbach basis}, which is a basis containing pairwise normal unit vectors with respect to the norm.
This leads to the following question.

\begin{problem}
Prove or disprove that if $\| \cdot \|_a = \| \cdot \|$ with respect to any symplectic form with a polar decomposition into an Auerbach basis of $\NS$, then $\NS$ is Euclidean.
\end{problem}

\begin{problem}
Characterize the norms $\| \cdot \|$ satisfying $\| \cdot \|_a = \| \cdot \|$ with respect to some symplectic form.
\end{problem}

Note that the normality map $J$ depends on the choice of the symplectic form $\langle\cdot, \cdot\rangle$ on $\X$.

\begin{ques}
Let $\NS$ be of dimension $2n > 2$. Do there exist symplectic forms with respect to which $J \Sph$ is the isoperimetrix of $\NS$
in the sense of Busemann or Holmes-Thompson?
\end{ques}

For these two concepts of isoperimetrices see, e.g. Chapter 5 of \cite{Th}, or \cite{Ma-Mu}.


\begin{ques}
We have shown in Theorem~\ref{pro7} that for any $M \in \XX$, we have $(M_\circ)^\circ = M$.
Clearly, $M \subseteq (M^\circ)_\circ$ also holds. Is it true that $M = (M^\circ)_\circ$?
\end{ques}

The requirements that the underlying normed space $(\X, \|\cdot\|)$  is smooth, strictly convex and of even dimension are not necessary
for the definition of semi-polarities; these requirements  are only  needed for the purpose that the normality map $J$ is well defined.

\begin{ques}
Is there a counterpart of Theorem~\ref{thm1} for odd dimensional spaces?
\end{ques}


\end{document}